\documentclass[12pt]{article}
\usepackage{amssymb,amsmath}
\usepackage{graphicx}
\usepackage{amsfonts}
\usepackage{float}
\usepackage{flafter}
\usepackage{color}

\labelwidth\leftmargini\advance\labelwidth-\labelsep

\def\R{\relax\ifmmode I\!\!R\else$I\!\!R$\fi}

\def\Z{\relax\ifmmode Z\!\!\!Z\else$Z\!\!\!Z$\fi}

\def\C{\relax\ifmmode C\!\!\!\!I\else$C\!\!\!\!I$\fi}

\def\K{\relax\ifmmode I\!\!K\else$I\!\!K$\fi}

\def\N{\relax\ifmmode I\!\!N\else$I\!\!N$\fi}

\newcounter{defcounter}[section]
{\vspace{0.1cm}\begin{sloppypar}\noindent\stepcounter{defcounter}{\bfseries
Definition
      \thesection.\thedefcounter}}%
{\end{sloppypar}\vspace{0.1cm}}

\newtheorem{corollary}{Corollary}[section]

\newtheorem{theorem}{Theorem}[section]

\newtheorem{proposition}{Proposition}[section]

\newcommand{\proof}{{\bf Proof.} }

\newcommand{\qed}{\hfill $\square$}

\begin{document}
\thispagestyle{empty}
\begin{center}
{\Large {\bf Multifractal aspects of $\alpha$-expansions}}
\end{center}
\begin{center}J\"org Neunh\"auserer\\
Technical University of Braunschweig \\
joerg.neunhaeuserer@web.de
\end{center}
\begin{center}
\begin{abstract}
In \cite{[NE]} we introduce $\alpha$-expansions a real numbers in $(0,1]$, given by
\[ \sum_{i=1}^{\infty}(\alpha-1)^{i-1}\alpha^{-(d_{1}+\dots+d_{i})}\]
with $\alpha>1$ and $d_{i}\in\mathbb{N}$ and discuss ergodic theoretical and dimension theoretical aspects of this expansions. In this sequel we study mutifractal aspects of this expansions.  
\end{abstract}
\end{center}
\section{Introduction}
Expansions of real numbers, like $b$-adic expansions and continued fraction expansions, are a central topic in the metric theory of numbers. 
The classical theory studies ergodic theoretical and dimension theoretical aspects of this expansion, see  \cite{[BO],[KIN],[HA],[BE],[EG],[JA]}. Nowadays the multifractal analysis of such expansions is one focus of research. In multifractal analysis dimension spectra of level sets, given by constrains on the sequence of digits, are studied. For example the level set studied are given by the arithmetic mean of digits of given growth rates of digits. We have dimensional theoretical and multifractal results for continued fraction expansions \cite{[FAN]},  $\beta$-expansions \cite{[BA],[BL]}, Cantor series expansions  \cite{[ER],[KIF],[DY]}, Engel-expansions  \cite{[WU],[FANG]}, Lüroth-expansions  \cite{[BAR],[FENG]}, multiple-base expansions \cite{[NE3],[NE3],[LI],[KO]} and others. \\
In \cite{[NE]} we introduce a new family of expansions of real numbers which is given essentially by the expression in the abstract. We call these expansions $\alpha$-expansion and developed the classical theory for them expansions. In this sequel we develop the multifractal analysis of $\alpha$-expansions. Notations and main results are given in the next section. In the following three sections we prove our results. 
\section{Notations and main results} 
Let $\alpha>1$ be a real number. In \cite{[NE]} we have shown that every $x\in (0,1]$ can by uniquely written in the form
\[ \sum_{i=1}^{\infty}(\alpha-1)^{i-1}\alpha^{-(d_{1}(x)+\dots+d_{i}(x))} \]  
with $(d_{i}(x))\in\mathbb{N}^{\mathbb{N}}$. We call this expansion $\alpha$-expansion of $x$. It may generated by the transformation
$T:(0,1]\to(0,1]$ 
\[ T(x)=\frac{\alpha^{i}}{\alpha-1}x-\frac{1}{\alpha-1}\mbox{ for }x\in \left(\frac{1}{\alpha^{i}},\frac{1}{\alpha^{i-1}}\right]\]
for $i\in\mathbb{N}$ via
\[ d_{i}(x)=\lceil-\log_{\alpha}(T^{i-1}(x))\rceil.\]
The local inverses of $T$ given by the family of contractions $T_{i}:(0,1]\to(0,1]$ with 
\[ T_{i}(x)=\frac{\alpha-1}{\alpha^{i}}x+\frac{1}{\alpha^{i}}\]
for $i\in\mathbb{N}$ define an infinit conformal iterated function system on $(0,1]$, see \cite{[MIH]}. 
The cylinder sets of this iterated function system are given by the left-side open intervals
\[ I(d_{1},\dots,d_{n})=\{x\in(0,1]|d_{i}(x)=d_{i},i=1,\dots,n\}\]
with length
\[  |I(d_{1},\dots,d_{n})|= (\alpha-1)^{n}\alpha^{-(d_{1}+\dots+d_{n})}.\]                                                             
Moreover we have
\[\langle d_{i}\rangle=\langle d_{1},d_{2},\dots\rangle:=\lim_{n\to\infty}T_{d_{n}}\circ \dots\circ T_{d_{1}}(1) \]
\[ =\sum_{i=1}^{\infty}(\alpha-1)^{i-1}\alpha^{-(d_{1}+\dots+d_{i})} \]
for a sequences $(d_{i})\in\mathbb{N}^{\mathbb{N}}$. This means $\langle d_{i}(x)\rangle=x$. \\
We have shown in \cite{[NE]} that the Lebesgue-measure is ergodic with respect to $T$. Hence the set
\[ U=\{x|(d_{i}(x))\mbox{ is unbounded}\}\]
has full Lebesgue-measure and by developing the dimension theory of $\alpha$-expansions we have shown that the set       
\[ B=\{x|(d_{i}(x))\mbox{ is bounded}\}\]
has Hausdorff dimension one, also it is of Lebesgue measure zero. In \cite{[NE]}
we missed the following interesting dimensional theoretical result. 
\begin{theorem} We have
\[ \dim_{H}\{x\in(0,1]|\lim_{i\to\infty}d_{i}(x)=\infty\}=0.\]
\end{theorem}
This theorem is opposed to the result of Good \cite{[GO]}, who proved that the set has Hausdorff dimension $1/2$ for the continued fractions expansion of irrational real numbers in $(0,1]$. The reader will find the proof of theorem 2.1 in the next section. \\
For $\beta >1$ we now set\[ M_{\beta}=\{x\in(0,1]|\lim_{k\to\infty}\frac{1}{k}\sum_{i=1}^{k}d_{i}(x)=\beta\}.\]
Using the ergodic theorem we proved in \cite{[NE]} that $M_{\alpha/(1-\alpha)}$ has full Lebesgue-measure. Using techniques developed in \cite{[FAN],[HA],[WANG]}, we find here continuous dimension spectrum.
 \begin{theorem} For all $\beta>1$ we have
\[\dim_{H}M_{\beta}=\frac{(1-\beta)\log(\beta-1)+\beta \log(\beta)}{-\log(\alpha-1)+\beta \log(\alpha)}\]  
\end{theorem} 
We define the Khintchine-spectrum $\kappa:(1,\infty)\to \mathbb{R}$ by
\[ \kappa(x)=\dim_{H}M_{x}= \frac{(1-x)\log(x-1)+x\log(x)}{-\log(\alpha-1)+x\log(\alpha)}.                  \]
Using elementary analysis theorem 2.3 implies the following corollary.
\begin{corollary}
$k$ is a positive differentiable unimodal function with maximum $(\alpha/(1-\alpha),1)$ and asymptotic
\[ \lim_{x\to 0}\kappa(x)=\lim_{x\to \infty}\kappa(x)=0.\]
\end{corollary}
\begin{figure}
\vspace{0pt}\hspace{0pt}\scalebox{0.42}{\includegraphics{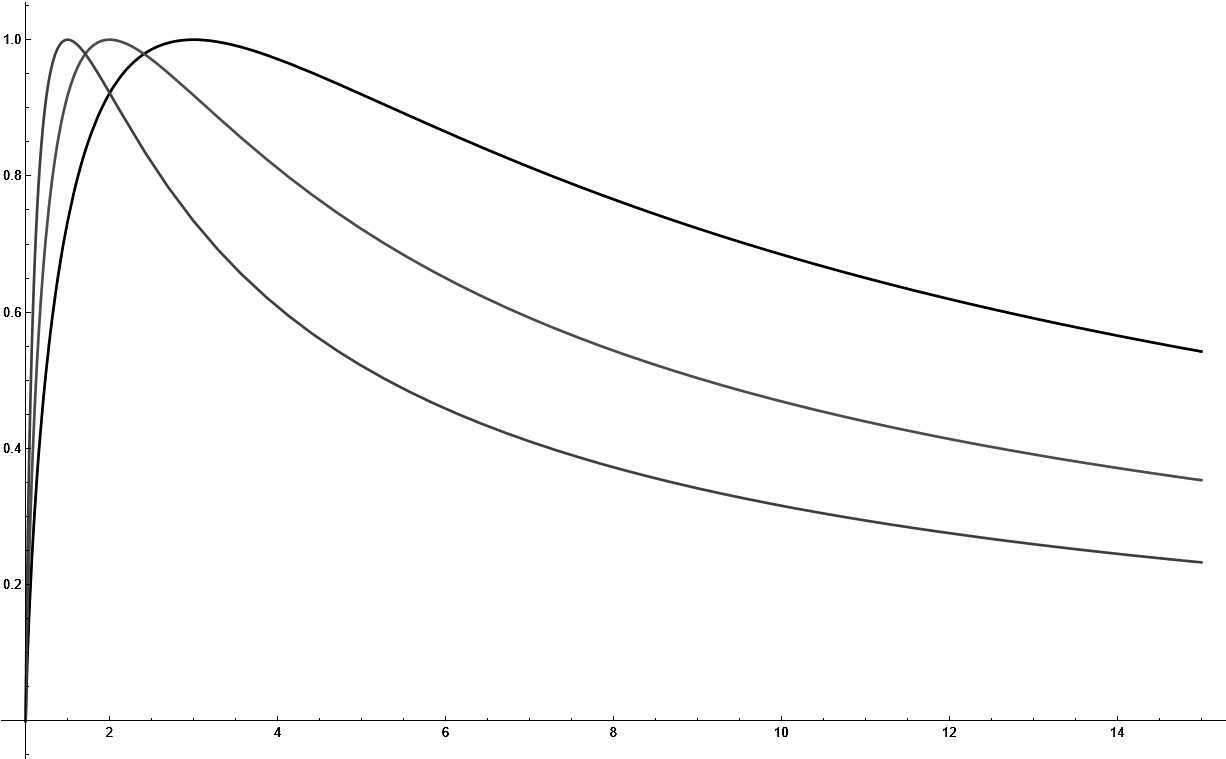}}
\caption{The Khintchine-spectrum $\kappa$ for $\alpha\in\{3/2,2,3\}$ }
\end{figure}
We display $\kappa$ for $\alpha\in\{3/2,2,3\}$ in figure 1.\\
In the last section of the paper we construct a multifractal spectrum in set $U=\{x|(d_{i}(x))\mbox{ is unbounded}\}$ using given increasing subsequences of digits. 
\begin{theorem}
Let $(m_{i})$ be a strictly increasing sequence of natural numbers such that $I=\{m_{i}|i\in\mathbb{N}\}$ has an infinite complement $\hat I=\mathbb{N}\backslash I$ and let $f:I\to\mathbb{N}$ be an increasing function. For $n\in\mathbb{N}$ choose $k_{n}$ such that   
\[ |\hat I\cap\{1,\dots k_{n}\}|=n. \]
If 
\[ \lim_{n\to\infty}\frac{k_{n}}{n}=1\quad\mbox{ and }\quad \lim_{n\to\infty}\frac{1}{n}\sum_{i=1}^{k_{n}-n}f(m_{i})=\mu\]
than
\[\dim_{H}\{x\in(0,1]|d_{m_{i}}(x)=f(i)~\forall i\in \mathbb{N}\}=d, \]
where $d\in (0,1)$ is the solution of 
\[\left(\frac{(a-1)}{\alpha^{\mu}}\right)^{d} -\alpha^{d}+1=0.\]    
\end{theorem}
Let us consider an example. Choose a sequence $(m_{i})$ of natural numbers such that  
$ k_{n}=n+\lceil\sqrt{2\mu n}\rceil$ and $f(m_{i})=i$. We have 
 \[\lim_{n\to\infty}\frac{1}{n}\sum_{i=1}^{k_{n}-n}f(m_{i})=\lim_{n\to\infty}\frac{1}{n}\sum_{i=1}^{k_{n}-n}i
 =\lim_{n\to\infty}\frac{(k_{n}-n)^2+(k_{n}-n)}{2n}=\mu.\]
Hence
\[\dim_{H}\{x\in(0,1]|d_{m_{i}}(x)=i~\forall i\in \mathbb{N}\}=d,\]
where $d$ is the solution of
\[\left(\frac{(a-1)}{\alpha^{\mu}}\right)^{d} -\alpha^{d}+1=0.\]   
We display the dimension-spectrum $d(\mu)$, given by this equation for $\alpha\in\{3/2,2,3\}$, in figure 2.
\begin{figure}
\vspace{0pt}\hspace{0pt}\scalebox{0.42}{\includegraphics{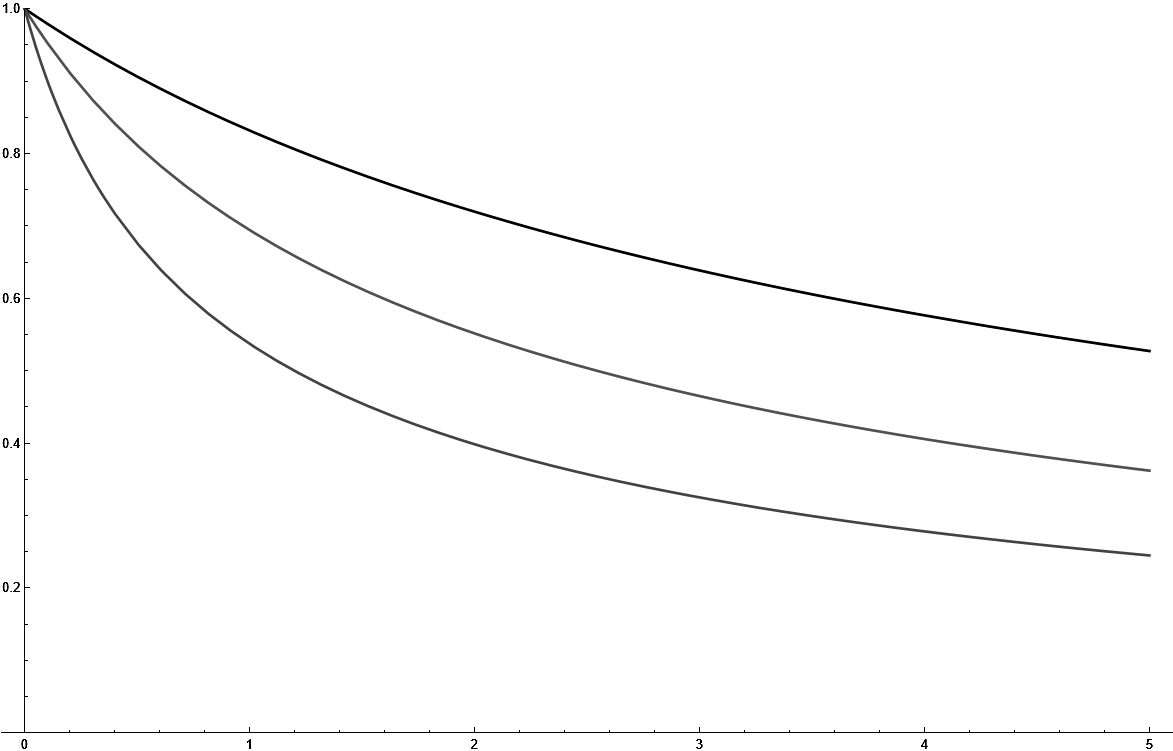}}
\caption{The dimension-spectrum $d(\mu)$ for $\alpha\in\{3/2,2,3\}$ }
\end{figure}     
\section{Proof of theorem 2.1}
We show in this section that the set
\[ A=\{x\in(0,1]|\lim_{i\to\infty}d_{i}(x)=\infty\}\]
has Hausdorff dimension zero. To this end we define a set
\[ A_{M}=\{x\in(0,1]|d_{i}(x)\ge M~\forall i\in\mathbb{N}\}\]
for an integer $M\ge 1$. Obviously $\dim_{H}A_{1}=1$, moreover we obtain:
\begin{proposition}We have
\[\lim_{M\to\infty}\dim_{H}A_{M}=0. \]
\end{proposition}
\proof By theorem 4.1 of \cite{[NE]} for all $M\in\mathbb{N}$ $\dim_{H}A_{M}$ is given by the unique solution $D\in(0,1]$ of the equation
\[ \sum_{i=M}^{\infty}(\frac{\alpha-1}{\alpha^{i}})^{D}=1.\]
The asymptotics for $M\to\infty$ is obvious.  
\qed\\~\\
$\dim_{H}A=0$ now follows from the next proposition. 
\begin{proposition}
For all $M\ge 1$ we have $\dim_{H}A\le\dim_{H}A_{M}$.
\end{proposition}
\proof
Fix $M\ge 1$. For all integers $N\ge 1$ we definite
\[ A_{M,N}=\{x\in(0,1]|d_{i}(x)\ge M~\forall i>N\}.\]
We have
\[ A_{M,N}=\bigcup_{s_{1},\dots,s_{N}\in\mathbb{N}}\{x\in(0,1]|d_{i}(x)\ge M~\forall i>N,\quad d_{i}(x)=s_{i}~\forall i\le N\}.\]
For $T_{i}(x)=\frac{\alpha-1}{\alpha^{i}}x+\frac{1}{\alpha^{i}}$  we have
\[ T_{s_{1}}\circ\dots\circ T_{s_{n}}(A_{M})=\{x\in(0,1]|d_{i}(x)\ge M~\forall i>N,\quad d_{i}(x)=s_{i}~\forall i\le N\}.\]
 Since a linear bijection does not change Hausdorff dimension, all these sets have dimension $\dim_{H}A_{M}$. Since Hausdorff dimension is countable stable, we obtain $\dim_{H}A_{M,N}=\dim_{H}A_{M}$.\\
For $x\in A$ we have $\lim_{i\to\infty}d_{i}(x)=\infty$, hence there is a $N$ such that $d_{i}(x)\ge M$ $\forall i>N$, hence $x\in A_{M,N}$ for some $N$. This means
\[ A\subseteq \bigcup_{N=1}^{\infty}A_{M,N}.\]
Again by countable stability the Hausdorff dimension of the union is $\dim_{H}A_{M}$ and $\dim_{H}A\le\dim_{H}A_{M}$ follows by monotony of dimension.
\qed
\section{Proof of theorem 2.2}
For a function $f:(0,1]\to \mathbb{R}$ we define the Brikhoff sum of $f$ with respect to $T$ by
\[ S_{n}f(x)=\sum_{j=0}^{n-1}f(T^{j}(x)),\]
where $T$ is the map that generates the $\alpha$-expansion, see section 2. For a given potential $\phi: (0,1]\to \mathbb{R}$ we define the topological pressure by 
\[ P_{\phi}(t,q)=\lim_{n\to\infty}\frac{1}{n}\log \sum_{(d_{1},\dots,d_{n})\in\mathbb{N}^{n}}\exp\left(\sup_{(d_{1}(x),\dots, d_{n}(x))=(d_{1},\dots,d_{n})}S_{n}(-t\log|T^{\prime}(x)|+q\phi(x))\right).\]
Fix the natural potential $\phi(x)=d_{1}(x)$. In this case we get an explicit expression for the topological pleasure.
\begin{proposition}
We have
\[ P_{\phi}(t,q)=\log\left(\frac{e^{-t\log(\alpha/(\alpha-1))+q}}{1-e^{-t\log(\alpha)+q}}\right).\]
\end{proposition}  
\proof We have 
\[ P_{\phi}(t,q)=\]
\[ =\lim_{n\to\infty}\frac{1}{n}\log \sum_{(d_{1},\dots,d_{n})\in\mathbb{N}^{n}}\exp\left(\sup_{(d_{1}(x),\dots, d_{n}(x))=(d_{1},\dots,d_{n})}\sum_{j=0}^{n-1}-t\log|T^{\prime}(T^{j}(x))|+qd_{j+1}(x)\right)\]
\[=\lim_{n\to\infty}\frac{1}{n}\log \sum_{(d_{1},\dots,d_{n})\in\mathbb{N}^{n}}\exp\left(\sum_{j=1}^{n}t\log((\alpha-1)/\alpha^{d_{j}})+qd_{j}\right)\]
\[=\lim_{n\to\infty}\frac{1}{n}\log \sum_{(d_{1},\dots,d_{n})\in\mathbb{N}^{n}}\exp\left(nt\log(\alpha-1)+(-t\log(\alpha)+q)\sum_{j=1}^{n}d_{j}\right)\]
\[ =\lim_{n\to\infty}\frac{1}{n}\log \left(\frac{e^{-t\log(\alpha/(\alpha-1))+q}}{1-e^{-t\log(\alpha)+q}}\right)^{n}\]
\[ = \log\left(\frac{e^{-t\log(\alpha/(\alpha-1))+q}}{1-e^{-t\log(\alpha)+q}}\right).\]
\qed~\\~\\
We will use the following corollary of the last proposition, which may be proved by a straightforward calculation
\begin{corollary}
For $\beta>1$ the solution of the system
\[ P(t,q)=q\beta \quad\mbox{ and }\quad  \frac{\partial P}{\partial q}(t,q)=\beta\]
is given by
\[ t(\beta)=\frac{(1-\beta)\log(\beta-1)+\beta \log(\beta)}{-\log(\alpha-1)+\beta \log(\alpha)}\quad\mbox{ and }\quad q(\beta)=t(\beta)\log(\alpha)+\log((\beta-1)/\beta).\]
\end{corollary} 
The family of local inverses of $T_{i}:(0,1]\to(0,1]$ for $i\in\mathbb{N}$ constitute an infinite linear and hence conform iterated function system in the sense of \cite{[MU]}. It is known that there is a Gibbs-measure for such iterated function systems:
\begin{theorem}
There is a unique shift invariant ergodic probability measure $\hat\mu_{t,q}$ on $\mathbb{N}^{N}$, such that the projected measure $\mu_{t,q}= \langle \hat\mu_{t,q}\rangle$ has the Gibbs-property
\[ \frac{1}{C}\le \frac{\mu_{t,q}(I(d_{1},\dots,d_{n}))}{\exp\left(\sum_{j=1}^{n}(t\log((\alpha-1)/\alpha^{d_{j}})+qd_{j})-nP(t,q)\right)}\le C\]
for all $n\in\mathbb{N}$ and $(d_{i})\in\mathbb{N}^\mathbb{N}$. Moreover
\[ \frac{\partial P}{\partial t}(t,q)=-\int\log|T^{\prime}(x)|d\mu_{t,q}\quad\mbox{ and }\quad\frac{\partial P}{\partial q}(t,q)=\int d_{1}(x)d\mu_{t,q}.\]
\end{theorem}  
The proof of the first part of this theorem can be found in section 2 of \cite{[HA]} and the second part of the theorem is proposition 4.8 of \cite{[FAN]}. Now we are prepared for the proof of theorem 2.2.\\\\
{\bf Proof of theorem 2.2} 
For $\beta>1$ choose $(t,q)=(t(\beta),q(\beta))$ according to corollary 4.1. Let $\mu_{t,q}$ be the corresponding measure from theorem 4.1. Since $\mu_{t,q}$ is ergodic with respect to $T$ for almost all $x\in(0,1]$ we have
\[\lim_{n\to\infty}\frac{1}{n}\sum_{i=1}^{n}d_{i}(x)=\lim_{n\to\infty}\frac{1}{n}\sum_{i=0}^{n-1}d_{1}(T^{i}(x))=\]
\[ = \int d_{1}(x)d\mu_{t,q}=\frac{\partial P}{\partial q}(t,q)=\beta.\]
Here we use the second part of theorem 4.1 and corollary 4.1.  It follows that $\mu_{t,q}(M_{\beta})=1$. For $x\in M_{\beta}$ we have
\[\lim_{n\to\infty}\frac{1}{n}\log(|I(d_{1}(x),\dots,d_{n}(x))|)=\lim_{n\to\infty}\frac{1}{n}\log((\alpha-1)^{n}\alpha^{-(d_{1}(x)+\dots+d_{n}(x))})\] 
\[ = \log(\alpha-1)-\beta \log(\alpha). \]
By theorem 4.1 and $P(t,q)=q\beta$, we obtain
\[ \lim_{n\to\infty}\frac{1}{n}\log(\mu_{t,q}(I(d_{1}(x),\dots,d_{n}(x))))= \lim_{n\to\infty}\frac{1}{n}\left(\sum_{j=1}^{n}(t\log((\alpha-1)/\alpha^{d_{j}})+qd_{j})-nP(t,q)\right) \]
\[ = \lim_{n\to\infty}\frac{1}{n}\left( \sum_{j=1}^{n}d_{j}(q-t\log(\alpha))+nt\log(\alpha-1)-nq\beta\right)=t\log(\alpha-1)-\beta t\log(\alpha).\]
Putting these equations together, we have
\[ \lim_{n\to\infty}\frac{\log(|I(d_{1}(x),\dots,d_{n}(x))|)}{\log(\mu_{t,q}(I(d_{1}(x),\dots,d_{n}(x))))}=t.\]
Billingsley lemma, see 1.4.1 in \cite{[BI]}, implies $\dim_{H}M_{\beta}=t=t(\beta)$ and this is theorem 2.2.\qed 
\section{Proof of theorem 2.3}
In this section we will prove theorem 2.3. Let $(m_{i})$ be a strictly increasing sequence of natural numbers such that $I=\{m_{i}|i\in\mathbb{N}\}$ has an infinite complement $\hat I=\mathbb{N}\backslash I$ and let $f:I\to\mathbb{N}$ be a increasing function. We consider the set
\[ B_{M}=B_{M}(I,f)=\{x\in(0,1]|d_{i}(x)=f(i)~\forall i\in I,~ d_{i}(x)\le M~\forall i\in\hat I\}.\]
The digits of $x$ in this set are given by $f(i)$ for an Index $i\in I$ and are bounded by $M$ for other indices.
For $n\in \mathbb{N}$ choose $k(n)$ such that the cardinality of $\hat I\cap\{1,2,\dots,k(n)\}$ is $n$. 
\begin{proposition} If  
\[ \lim_{n\to\infty}\frac{k_{n}}{n}=1\quad\mbox{ and }\quad \lim_{n\to\infty}\frac{1}{n}\sum_{i=1}^{k_{n}-n}f(m_{i})=\mu\]
the Hausdorff dimension $d\in (0,1)$ of $B_{M}$ is given the solution of
\[ \sum_{i=1}^{M}\left(\frac{(\alpha-1)}{\alpha^{i+\mu}}\right)^d=1- \]
\end{proposition}
\proof For a sequence $(d_{i})\in\mathbb{N}^{n}$ we define
 \[  \hat I(d_{1},\dots,d_{n})=\]
 \[ \{x\in(0,1]|d_{i}(x)=d_{i}\forall i\in(\mathbb{N}\backslash I)\cap\{1,2,\dots,k(n)\},~d_{i}(x)=f(i)\forall i\in I\cap\{1,2,\dots,k(n) \}\}.\]
We define a probability measure $\mu$ on $B_{M}$, by
\[\mu(\{x\in(0,1]|d_{i}(x)=j\})=\left(\frac{(\alpha-1)}{\alpha^{j+\mu}}\right)^{d}\] for $i\in I\backslash\mathbb{N}$ and \[\mu(\{x\in(0,1]|d_{m_{i}}(x)=f(i)\})=1,\quad \mu(\{x\in(0,1]|d_{m_{i}}(x)\not=f(i)\})=0 \] for $m_{i}\in I$.
The measure of the intervals $\hat I$ is given by
\[ \mu(\hat I(d_{1},\dots,d_{n}))=\prod_{i=1}^{n} \left(\frac{(\alpha-1)}{\alpha^{d_{i}+\mu}}\right)^{d}\]
and the length of the intervals is given by
\[ |\hat I(d_{1},\dots,d_{n})|=\prod_{i=1}^{n}\frac{\alpha-1}{\alpha^{d_{i}}}\prod_{i=1}^{k_{n}-n} \frac{\alpha-1}{\alpha^{f(m_{i})}}.\]
Hence
\[ \lim_{n\to\infty}\frac{\log(\mu(\hat I(d_{1},\dots,d_{n})))}{\log(|\hat I(d_{1},\dots,d_{n})|)}\]
\[ =d\lim_{n\to\infty}\frac{n\log(\alpha-1)- (n\mu+\sum_{i=1}^{n}d_{i})\log(\alpha)}{k_{n}\log(\alpha-1)-(\sum_{i=1}^{k_{n}-n}f(m_{i})+\sum_{i=1}^{n}d_{i})\log(\alpha)}=1\]
by our assumptions. For every $x\in B_{M}$ there is a sequence $(d_{i})\in\mathbb{N}^{\mathbb{N}}$ such that $x\in\hat I(d_{1},\dots,d_{n})$ for all $n\in\mathbb{N}$ hence the result follows from Billingsley lemma, see lemma 1.4.1 in \cite{[BI]}.  \qed\\
The following corollary of proposition 5.1 is straightforward.
\begin{corollary}
Under the assumption of proposition 5.1, we have $\lim_{M\to\infty}\dim_{H}B_{M}=d$, where $d\in(0,1)$ is given by
\[\left(\frac{(a-1)}{\alpha^{\mu}}\right)^{d} -\alpha^{d}+1=0.\] 
\end{corollary} 
\proof
Let $\dim_{H}B_{M}=d_{M}$ by the last proposition, we have
\[ \sum_{i=1}^{M}\left(\frac{(\alpha-1)}{\alpha^{i+\mu}}\right)^{d_{M}}=1. \]
Taking the limit $M\to\infty$ and using the geometric series, gives
\[ \frac{((a-1)a^{-m-1})^d}{1-a^{-d}}=1,\]
which is equivalent to the expression given in the corollary. 
\qed

\end{document}